\documentclass{article}

\usepackage{PRIMEarxiv}

\usepackage[utf8]{inputenc} 
\usepackage[T1]{fontenc}    
\usepackage{hyperref}       
\usepackage{url}            
\usepackage{booktabs}       
\usepackage{amsfonts}       
\usepackage{nicefrac}       
\usepackage{microtype}      
\usepackage{bm}

\usepackage{amsmath}
\usepackage{mathrsfs}
\usepackage{amssymb}

\usepackage{multirow} 
\usepackage{extarrows}

\usepackage{diagbox}
\usepackage{booktabs}
\usepackage{multirow}

\usepackage{lipsum}

\usepackage{floatrow}

\usepackage{environ}
\usepackage{tikz}

\usepackage{fancybox}
\NewEnviron{elaboration}{
\par
\begin{tikzpicture}
\node[rectangle,minimum width=0.5\textwidth] (m) {\begin{minipage}{0.79\textwidth}\BODY\end{minipage}};
\draw[dashed] (m.south west) rectangle (m.north east);
\end{tikzpicture}
}

\newfloatcommand{capbtabbox}{table}[][\FBwidth]
\DeclareSymbolFont{largesymbol}{OMX}{yhex}{m}{n}
\DeclareMathAccent{\Widehat}{\mathord}{largesymbol}{"62}

\newtheorem{theorem}{Theorem}

\newtheorem{assumption}{Assumption}

\newtheorem{proof}{proof}

\usepackage{lipsum}
\usepackage{fancyhdr}       
\usepackage{graphicx}       
\graphicspath{{media/}}     

\title{Bounding the largest intersection number of regular simplicial partitions}

\author{
  Zhiheng Zhang \\
  Institute for Interdisciplinary Information Sciences, Tsinghua University, \\
  \texttt{zhiheng-20@mails.tsinghua.edu.cn} \\
}

\begin{document}
\maketitle

\begin{abstract}
In this note, we first identify the maximum intersection number of the regular simplicial partitions.
\end{abstract}


\section{Introduction}
The concept of simplicial partitions is common and useful in Finite element analysis~\cite{segerlind1991applied, bhavikatti2005finite}. Simplicial partition tends to cut the space being searched into regularised chunks, finding the optimal solution to the objective function over the whole space by iteration, e.g., branch-and-bound~\cite{lawler1966branch}. In this process, it is particularly important what geometry the simplex can maintain, as this can seriously interfere with the efficiency of the solution. For example, if, after enough partitions, the simplex degenerates approximately to a hyperplane, there is no guarantee that the theoretical optimal point of the objective function will belong to any simplex (it may be wrapped in many or even an infinite number of simplices). This uncertainty about the region to which the solution belongs can make the solution to the optimization problem significantly more expensive.

There is already a range of widely adopted conditions to constrain the regularity of simplicial partitions. In~\cite{brandts2009equivalence}, the authors summarized three kinds of equivalent conditions, which are called balled conditions. Meanwhile, there are also other conditions such as minimum angle conditions as in~\cite{ZLAMAL1968}. People have already demonstrated the minimum angle condition is equivalent to the three kinds of balled conditions above in the 3-dimensional cases. In the high-dimensional cases, authors~\cite{brandts2011generalization} developed a generalized version of the minimum angle condition. 

\section{Problem Formulation}
In this paper, we directly quote one of the most widely-used balled conditions to define the regularity, the other two balled condition~\cite{ZLAMAL1968} are both equivalent.

\begin{assumption}{\textbf{(Regularity assumption~\cite{brandts2009equivalence, ZLAMAL1968, brandts2011generalization})}}
For each simplex $S$ in simplicial partitions, there exists a constant $\eta$ such that
\begin{equation}
    \begin{aligned}
    vol(S) \geq \eta {h(S)}^{d}.
    \end{aligned}
\end{equation}
Here $vol(S)$ denotes the volumn of simplex $S$, $h(S)$ denotes the longest edge of $S$, and $d\geq 2$ is the dimension.

\label{ass}
\end{assumption}

Roughly speaking, the regularity assumption guarantees the simplex would not degenerate to the hyper-plane, or else $\lim\limits_{vol(S)\rightarrow 0} \frac{vol(S)}{{h(S)}^d} = 0$. In other words, $\frac{vol(S)}{{h(S)}^d}$ can be higher when the simplex ``seems to be regular'', namely each edge keeps the same length. In addition, the simplicial partitions during partitioning are denoted as $\mathcal{S}_0, \mathcal{S}_1, \mathcal{S}_2...$. Our problem can be summarized as follows: 

\emph{what is the maximum intersection number of regular simplicial partitions? In other words, for an arbitrary point in $\mathcal{S}_0$, what is the maximum number of simplices it can be affiliated with during partitioning?}

\section{Main result}
In this part, we present our main result as follows:

\begin{theorem} 
Suppose that Ass.~\ref{ass} holds. Each point in $\mathcal{S}_0$ is included within at most $\frac{1}{\eta}\left(\frac{2e \pi}{ d}\right)^{\frac{d}{2}}$ simplices. 
 
\end{theorem}

\begin{proof}
For $\gamma_0 \in \mathcal{S}_0$, we construct a ball $\mathcal{B}(\gamma_0, r)$. We use $A(\cdot)$ to denote the surface area of the sphere:
\begin{equation}
    \begin{aligned}
    A\left(\mathcal{B}\left(\gamma_0, r\right) \cap \mathcal{S}_0\right)=\sum_{S \in \mathcal{S}_k, \gamma_0 \in S} A\left(\mathcal{B}\left(\gamma_0, r\right) \cap {S} \right).
    \end{aligned}\label{original}
\end{equation}
On the one hand, the LHS of Eqn.~\ref{original} can be upper bounded as:
\begin{equation}
    \begin{aligned}
    A\left(\mathcal{B}\left({\gamma_0}, r\right) \cap S_0\right) \leq A\left(\mathcal{B}\left({\gamma_0}, r\right)\right)=\frac{2 \pi^{\frac{ d}{2}}}{\Gamma\left(\frac{ d}{2}\right)} r^{d-1}<+\infty .
    \end{aligned}\label{A_S_0}
\end{equation}
On the other hand, we calculate the RHS of Eqn.~\ref{original}. To solve it, we take
advantage of the following integral (we take $\gamma_0$ as the origin):
\begin{equation}
    \begin{aligned}
    \int_{\mathcal{W}_S \gamma \geq 0} e^{-\|\gamma\|^2} d \gamma.
    \end{aligned}
\end{equation}
Here for each term on the right side, the sphere of the ball is cut by certain facets of each S, whose
normal vector is denoted as a set $\mathcal{W}_S$, whose each row denotes a $1 * d$ normal vector. Without loss of generation, for any facets, we assume that its normal vector points to the remaining
supporting vector, namely its inner product is positive.

Here $\| \cdot \|$ also denotes the Euclidean norm. If we use the polar coordinates, we can get a new expression by the differential element method:

\begin{equation}
    \begin{aligned}
    \int_{\mathcal{W}_S \gamma \geq 0} e^{-\|\gamma\|^2} d \gamma=\int_{\mathcal{W}_S \gamma \geq 0} d \Omega \int_0^{+\infty} e^{-l^2} l^{ d-1} d l=\frac{A\left(\mathcal{B}\left(\gamma, r\right) \cap S\right)}{r^{ d-1}} \int_0^{+\infty} e^{-l^2} l^{ d-1} d l.
    \end{aligned}\label{A_S}
\end{equation}
According to Eqn.~\ref{A_S_0}-\ref{A_S}, we have
\begin{equation}
    \begin{aligned}
A\left(\mathcal{B}\left(\gamma_0, r\right) \cap S\right) &=\frac{2 \int_{\mathcal{W}_S \gamma \geq 0} e^{-\|\gamma\|^2} d \gamma}{\Gamma\left(\frac{d}{2}\right)} r^{d-1} \\
\forall k, \frac{A\left(\mathcal{B}\left(\gamma_0, r\right) \cap S\right)}{A\left(\mathcal{B}\left(\gamma_0, r\right) \cap \mathcal{S}_0\right)} & \geq \frac{\int_{\mathcal{W}_S \gamma \geq 0} e^{-\|\gamma\|^2} d \gamma}{\pi^{\frac{d}{2}}}.
    \end{aligned}
\end{equation}
Hence we only need to prove the integral is lower bounded by a constant above zero. We will do this by extracting a sub-space from $\mathcal{W}_S \gamma \geq 0$, which is easy to be integrated. Specifically, we consider a subspace which is an affine transformation on $\mathcal{W}_S \gamma \geq  0$. We introduce the diameter of simplex $S$ as $\operatorname{dia}(S) = \max _{s_1, s_2 \in S}\left\|s_1-s_2\right\|$.

\begin{equation}
    \begin{aligned}
\int_{\mathcal{W}_S \gamma \geq 0} e^{-\|\gamma\|^2} d \gamma & \geq \int_{\gamma=t \gamma^{\prime}, \gamma^{\prime} \in S} e^{-\|\gamma\|^2} d \gamma \quad(\forall t>0, t \text { is arbitrarily chosen }) \\
&=\int_{\gamma^{\prime} \in S} t^{d} e^{-t^2\left\|\gamma^{\prime}\right\|^2} d \gamma^{\prime} \\
& \stackrel{*}{\geq } t^d \operatorname{Vol}(S) e^{-t^2 \operatorname{dia}(S)^2} \\
& \stackrel{* *}{\geq} \eta t^{ d} h(S)^{ d} e^{-t^2 \operatorname{dia}(S)^2}
\end{aligned}
\end{equation}
$*$ is due to $\forall \gamma^{\prime} \in S$, we have $\left\|\gamma^{\prime}\right\| \leq \operatorname{dia}(S)$, since $\gamma_0$ is chosen as the origin. $* *$ is due to Ass.~\ref{ass}. Additionally, we further show that in the above Formulation, $h(S)=\operatorname{dia}(S)$. Namely, the longest edge of simplex always serves as the diameter. It is equal to prove ($S^i$ denotes the supporting vector):
\begin{equation}
    \begin{aligned}
\operatorname{dia}(S) &=\max _{s_1, s_2 \in S}\left\|s_1-s_2\right\| \quad\left(s_2=\sum_i \lambda_i S^i, \lambda_i \in[0,1]\right) \\
&=\max _{s_1, s_2 \in S}\left\|\left(\sum_i \lambda_i\right) s_1-\sum_i\left(\lambda_i S^i\right)\right\| \\
&=\max _{s_1, s_2 \in S}\left\|\sum_i \lambda_i\left(s_1-S^i\right)\right\| \\
& \leq \max _{s_1 \in S} \max _i\left\|s_1-S^i\right\| \\
& \stackrel{*}{\leq} \max _{i, j}\left\|S^j-S^i\right\| \leq h(S).
\end{aligned}
\end{equation}
$*$ is due to we also do the expansion on $s_1$, namely $s_1=\sum_i \lambda_i^{\prime} S^i, \lambda_i^{\prime} \in[0,1]$. On the other hand, we have $h(S) \leq \max _{s_1, s_2 \in S}\left\|s_1-s_2\right\|=\operatorname{dia}(S)$ by definition. Thus $\operatorname{dia}(S)=h(S)$. Hence
\begin{equation}
    \begin{aligned}
    \forall t, \int_{\mathcal{W}_S \boldsymbol{\gamma} \geq 0} e^{-\|\gamma\|^2} d \boldsymbol{\gamma} \geq \eta(t d i a(S))^{d} e^{-\left(t \operatorname{dia}(S)^2\right)}.
    \end{aligned}
\end{equation}

Due to the arbitrary of $t$, we have
\begin{equation}
    \begin{aligned}
    \int_{\mathcal{W}_S \boldsymbol{\gamma} \geq 0} e^{-\|\gamma\|^2} d \boldsymbol{\gamma} \geq \eta \max _x x^{d} e^{-x^2}=\left.\eta x^{d} e^{-x^2}\right|_{x=\sqrt{\frac{d}{2}}}=\eta(\frac{d}{2})^{\frac{d}{2}} e^{- \frac{d}{2}}.
    \end{aligned}
\end{equation}
Finally, we have
\begin{equation}
    \begin{aligned}
    \forall k, \frac{A\left(\mathcal{B}\left(\gamma, r\right) \cap S\right)}{A\left(\mathcal{B}\left(\gamma, r\right) \cap S_0\right)} \geq \eta\left(\frac{d}{2e \pi}\right)^{\frac{d}{2}}, \text{the~intersection~number~} N \leq \frac{1}{\eta}\left(\frac{2e \pi}{ d}\right)^{\frac{d}{2}}<+\infty.
    \end{aligned}
\end{equation}

\end{proof}
In future work, I will further explore the tight upper bound of the intersection number of the regular simplicial partitions.

\section{Acknowledgments}
This note is my additional exploration after the work~\cite{zhihengproxy}. 

\bibliographystyle{unsrt}  
\bibliography{references}

\end{document}